\documentclass[10pt]{amsart}
\usepackage{amssymb, amsmath}
\usepackage{graphics,xcolor}
\usepackage{latexsym,amsmath,amssymb,amsthm,amsfonts,mathrsfs}
\usepackage{amscd}
\usepackage[arrow, matrix, curve]{xy}

\theoremstyle{plain}
\newtheorem{thm}{Theorem}[section]
\newtheorem{theorem}[thm]{Theorem}
\newtheorem{lemma}[thm]{Lemma}
\newtheorem{corollary}[thm]{Corollary}
\newtheorem{proposition}[thm]{Proposition}

\theoremstyle{definition}

\newtheorem{question}[thm]{Question}

\newtheorem{remark}[thm]{Remark}

\newtheorem{definition}[thm]{Definition}

\numberwithin{equation}{thm}

\newcommand{\C}{{\mathbb C}}

\title[]{On Shimura subvarieties of the Prym locus}
\author{Abolfazl Mohajer}
\address{Universit\"{a}t Mainz, Fachbereich 08, Institut f\"ur Mathematik, 55099 Mainz, Germany}
\email{mohajer@uni-mainz.de}
\subjclass{14G35, 14H15, 14H40}
\keywords{Shimura variety, Prym variety}

\begin{document}

\begin{abstract}
We show that families of Pryms of  abelian Galois covers of
$\mathbb{P}^{1}$ in $A_{g}$ do not give rise to  high dimensional Shimura
subvareties. Our method is based on analyzing the Monodromy of the family.
\end{abstract}

\maketitle

\section{Introduction}

In \cite{CFGP}, Shimura curves of PEL type in $A_g$, generically
contained in the Prym locus have been studied. Let $R_g$ be the
scheme of isomorphism classes $[C,\eta]$, for $C$ a smooth
projective curve of genus $g$ and $\eta\in Pic^0(C)$ is a
2-torsion element, i.e., $\eta\neq \mathcal{O}_C$ but
$\eta^2=\mathcal{O}_C$. A point $[C,\eta]$ corresponds to an
\'etale double cover $h:\tilde{C}\to C$. There is an induced norm
map $Nm:Pic^0(\tilde{C})\to Pic^0(C)$ and the Prym variety
associated to $[C,\eta]$ is defined to be connected component of
the identity of $\ker Nm$ and is denoted by $P(C,\eta)$ or
$P(\tilde{C},C)$. Similarly, one can construct the ramified Prym
variety. To do this, one considers the scheme $R_{g,2}$ of pairs
$[C,B,\eta]$ up to isomorphism such that $\eta$ of degree $1$ on
$C$ and $B$ a divisor in the linear series $|\eta^2|$
corresponding to a double covering $\pi:\tilde{C}\to C$ ramified
above $B$. The assignment $[C,B,\eta]\mapsto P(\tilde{C}, C)$
defines a map $\mathscr{P}:R_{g,2}\to A_g$. Both in the unramified
Prym locus corresponding to the double \'etale covers and the
ramified Prym locus, corresponding to the family of double covers
ramified at two points, \cite{CFGP} gives examples of Shimura
curves in the Prym locus provided that the quotient of the base
curve by the group is $\mathbb{P}^{1}$. In fact, they consider a
family of Galois covers $\tilde{C_t}\to \mathbb{P}^{1}$ with
Galois group $\tilde{G}$ and a central involution $\sigma$ and
also double covering $\tilde{C_t}\to \tilde{C_t}/\langle
\sigma\rangle$ which is either \'etale or ramified at exactly two
distinct points. By the theory of coverings, such a Galois
covering is given by an epimorphism $\tilde{\theta}:\Gamma_r\to
\tilde{G}$ with branch points $t_1,\cdots,t_r\in \mathbb{P}^{1}$.
Here $\Gamma_r$ is the braid group on $r$ elements which is
isomorphic to the fundamental group of $\mathbb{P}^{1}\setminus
\{t_1,\cdots,t_r\}$. Varying the branch points, we get a family
$R(\tilde{G},\tilde{\theta},\sigma)\subset R_g$. Note that in the
above mentioned paper, the families are all one-dimensional, so
that the examples all give rise to Shimura curves. Also, the
authors ask the following question

\begin{question}(\cite{CFGP})
Do there exist high dimensional Shimura (special) subvarieties contained
generically in the Prym locus?
\end{question}

In this paper, we investigate this
question when $\tilde{G}$ is abelian and show that in this case the answer to the above question is negative.
Note that in \cite{CF} the authors give upper bounds for the dimension of a germ of a totally 
geodesic submanifold, and hence of a special subvariety in the Prym locus. In \cite{MZ}, we have applied
characteristic $p$ methods also to exclude Shimura curves from the
Torelli locus. However, due to the fact that the Prym map is not
an embedding, this method does not work here. In section $2$, we
explain an alternative construction of the abelian covers of
$\mathbb{P}^{1}$ as given in \cite{CFGP} and also the construction
of the Prym variety and the Prym map for which we mostly follow
\cite{CFGP}. More precisely, we fix integers $N\geq 2$, $s\geq 4$,
$m\geq 1$ and an $m \times s$ matrix $A$ with entries in
$(\mathbb{Z}/N\mathbb{Z})$. Given an $s$-tuple
$t=(z_{1},...,z_{s})\in (\mathbb{A}^{1}_{\mathbb{C}})^{s}$ there
is a Galois cover $Y_{t}\rightarrow \mathbb{P}^{1}$ branched at
the points $z_{j}$ with local monodromies encoded in the matrix
$A$ and an abelian Galois group $\tilde{G}$ which is isomorphic to
the column span of the matrix $A$ and hence is a subgroup of the
group $(\mathbb{Z}/N\mathbb{Z})^{m}$. Using this construction, the $\sigma$-action and its
eigenspaces and also the eigenspaces of the whole group acting on this spaces that are useful
for our computations are more conceretely and more easily computable. By varying the branch points 
we obtain a family $f:Y\to T$ of abelian covers of $\mathbb{P}^{1}$. The image of $T$ in $R_g$ (resp. $R_{g,2}$),
which we also denote by $T$ is of dimension $s-3$ and we are
interested about the image $Z=\mathscr{P}(T)$. There exists a natural variation of Hodge structures
over $T$ that enables us to construct a special subvariety $S_f$ that contains $Z$. It
turns out, that this is the smallest special subvariety of $A_g$ with this property. So $Z=S_f$, or equivalently $\dim Z=\dim S_f$,
if and only if $Z$ is a special subvariety. Unlike the Torelli
map, the Prym map is not injective, however, it holds that $\dim
Z\leq s-3$.  Our strategy is therefore to show that if $s$ is
large, the special subvariety $S_f$ is of dimension strictly
greater than $s-3$ and thus showing that for large $s$, $\dim
Z<\dim S_f$ and hence $Z$ is not special, see Lemma~\ref{dimS_f}. In section 3, we prove
the above inequality. Our method, previously used in \cite{MZ} for
families of Jacobians, is analyzing the group structure of the
generic Mumford-Tate group, that results in inequalities for $\dim
S_f$ which we apply to prove the above statements. We discuss
monodromy of the higher dimension families (i.e., with large $s$)
and thereby show that for such families, the monodromy group is
large and prove as a consequence that the family does not give rise to a
Shimura subvariety. Since we use only the group structure of the
family of abelian covers, our results are insensitive to the
ramification behavior of the double covering, so we state and
prove our results usually without distinguishing between
unramified and ramified Prym locus.

\section{Construction of abelian covers of $\mathbb{P}^{1}$ and the Prym map}

In this section, we follow closely \cite{CFGP} and also \cite{MZ} whose notations come mostly from \cite{W}. More details
about abelian coverings and Prym varieties can be consulted from
these two references respectively. For the latter, \cite{BL} is also a
comprehensive reference.

An abelian Galois cover of $\mathbb{P}^{1}$ is determined by a
collection of equations in the following way: Consider an $m\times
s$ matrix $A=(r_{ij})$ whose entries $r_{ij}$ are in
$\mathbb{Z}/N\mathbb{Z}$ for some $N\geq 2$. Let
$\overline{\mathbb{C}(z)}$ be the algebraic closure of
$\mathbb{C}(z)$. For each $i=1,...,m,$ choose a function $w_{i}\in
\overline{\mathbb{C}(z)}$ with
\[w_{i}^{N}=\prod_{j=1}^{s}(z-z_{j})^{\widetilde{r}_{ij}}\text{ for }i=
1,\cdots, m,\] where $\widetilde{r}_{ij}$ is the lift of $r_{ij}$
to $\mathbb{Z} \cap [0,N)$. We impose the condition that the sum
of the columns of $A$ are zero (in $\mathbb{Z}/N\mathbb{Z}$). This
implies that the cover is \emph{not} ramified over infinity. We
call the matrix $A$, the matrix of the covering. We also remark
that all operations with rows and columns will be carried out in
the ring $\mathbb{Z}/N\mathbb{Z}$, i.e., they will be considered
modulo $N$. The local monodromy around the branch point $z_{j}$ is
given by the column vector $(r_{1j},....,r_{mj})^{t}$ and so the
order of ramification over $z_{j}$ is $
\frac{N}{gcd(N,\widetilde{r}_{1j},..,\widetilde{r}_{mj})}$. Using
this and the Riemann-Hurwitz formula, the genus $g$ of the cover
can be computed by:
\[g= 1+ d(\frac{s-2}{2}- \frac{1}{2N}\sum_{j=1}^{s}
gcd(N,\widetilde{r}_{1j},...,\widetilde{r}_{mj})),\] 
where $d$ is the degree of the covering which is equal, as pointed out above,
to the column span (equivalently row span)  of the matrix $A$. In
this way, the Galois group $G$ of the covering will be a subgroup
of $(\mathbb{Z}/N\mathbb{Z})^{m}$. Note also that this group is
isomorphic to the column span of the above matrix.

As explained in the introduction, giving a Galois cover of
$\mathbb{P}^{1}$ is equivalent to giving an epimorphism
$\theta:\Gamma_s\to G$, see \cite{Vo}.

Using the above construction, we can also construct families of
abelian covers of the line by letting the branch points vary on an
affine scheme. The details are as follows:\\

Let $T \subset (\mathbb{A}^{1})^{s}$ be the complement of the big
diagonals, i.e., $T=\mathcal{P}_{s}= \{(z_{1},....,z_{s})\in
(\mathbb{A}^{1})^{s}\mid z_{i}\neq z_{j} \forall i\neq j \}$. Over
this open affine set, we define a family of abelian covers of
$\mathbb{P}^{1}$ to have the equation:
\[w_{i}^{N}=\prod_{j=1}^{s}(z-z_{j})^{\widetilde{r}_{ij}}\text{ for }i=
1,\cdots, m,\]
where $(z_{1},...,z_{s})\in T$ and
$\widetilde{r}_{ij}$ is the lift of $r_{ij}$ to $\mathbb{Z}\cap
[0,N)$ as before. Varying the branch points we get a family
$f:C\to T$ of smooth projective curves over $T$ whose fibers $C_t$
are abelian covers of $\mathbb{P}^{1}$ introduced above.\\

Now let us explain the construction of the Prym locus of the above
family of abelian covers of the projective line. For the general
case, we refer to \cite{CFGP}.\\

Let $\tilde{C}\to \mathbb{P}^{1}$ be an abelian $\tilde{G}$-Galois
cover as constructed above. Let $\sigma\in \tilde{G}$ be an
involution and define $G=\tilde{G}/\langle \sigma\rangle$. Let
$\pi: \tilde{C}\to C$ be an element of the family and let $\eta\in
Pic^0(C)$ be the 2-torsion element yielding the \'etale double
covering $\pi$. Set $V=H^0(\tilde{C},K_{\tilde{C}})$ and let
$V=V_+\oplus V_-$ be the eigenspace decomposition for the action
of $\sigma$. There is also a corresponding decomposition for
$H^1(\tilde{C_t},\mathbb{C})_-=V_{-,t}\oplus \overline{V}_{-,t}$.
The $\sigma$-action also yields
$\Lambda=H_1(\tilde{C_t},\mathbb{Z})_-$. The associated Prym
variety is by definition the following abelian variety of
dimension $g-1$
\[P(C_t,\eta_t)=V^{*}_{-,t}/\Lambda,\]
see \cite{BL} for more details. Therefore, we get the Prym
map
\[\mathscr{P}:R_g\to A_{g-1},\]
where $R_g$ is as in the introduction.\\

Analogously we can consider the moduli space parametrising ramified 
double coverings and the corresponding Prym varieties. Let $R_{g,2}$, as
in the introduction, be the scheme $R_{g,2}$ of pairs
$[C,B,\eta]$ up to isomorphism such that $\eta$ of degree $1$ on
$C$ and $B$ a divisor in the linear series $|\eta^2|$
corresponding to a double covering $\pi:\tilde{C}\to C$ ramified
above $B$. The assignment $[C,B,\eta]\mapsto P(\tilde{C}, C)$
defines a map $\mathscr{P}:R_{g,2}\to A_g$. We remark that since we use only the group structure of the
Mumford-Tate group of the family, our results are insensitive to the
ramification behavior of the double covers, so we state and
prove our results usually for special subvarieties of $A_{g-1}$ whereas they are more generally true 
for ramified Prym locus.

\begin{definition}
Fix a matrix $A$ as in the beginning of this section. A Prym datum is a triple $(\tilde{G}, \tilde{\theta},\sigma)$
where $\tilde{G}$ is the abelian group generated by the columns of the matrix $A$ and $\sigma\in \tilde{G}$ is an element of order
2 that is not contained in $\displaystyle \cup_i <T_i>$ where $T_i$ is the $i$-th column of the matrix  $A$.
\end{definition} 

Since we will be dealing with special subavarities of $A_g$, the moduli space of
principally polarized abelian varieties of dimension $g$, we sketch the construction of $A_g$ as a
Shimura variety and refer to \cite{M} and \cite{CFGP} for more details. Let $V_{\mathbb{Z}}:=\mathbb{Z}^{2g}\subset
V_{\mathbb{Q}}:=\mathbb{Q}^{2g}$ and let $\psi:
V_{\mathbb{Z}}\times V_{\mathbb{Z}}\to \mathbb{Z}$ be the standard
symplectic form. Let $L=Gsp(V_{\mathbb{Z}},\psi)$ be the group of
symplectic similitudes and
$\mathbb{S}:=Res_{\mathbb{C}/\mathbb{R}}\mathbb{G}_m$ be the
Deligne torus. Consider the space $\mathcal{H}_g$ of homomorphisms
$h:\mathbb{S}\to L_{\mathbb{R}}$ that define a Hodge structure of
type $(1,0)+(0,1)$ on $V_{\mathbb{Z}}$ with $\pm (2\pi i)\psi$ as
a polarization. The pair $(L_{\mathbb{Q}},\mathcal{H}_g)$ is a
Shimura datum and $A_g$ can be described as the Shimura variety
associated to this Shimura datum as follows. Let $K_n:=\{g\in
G(\widehat{\mathbb{Z}})|g\equiv 1$ $(mod$ $n)\}$ with $n\geq 3$.
Then $A_{g,n}(\mathbb{C})=L(\mathbb{Q})\setminus
\mathcal{H}_g\times L(\mathbb{A}_f)/K_n$. The natural number $n$
is called the level structure.
Since it does not play any role in the arguments, we will omit it from the notation. 
As $A_g$ has the structure of a Shimura variety, one can define its special (or Shimura) subvarieties. Consider an algebraic
subgroup $N\subset L_{\mathbb{Q}}$ for which the set
\[Y_N=\{h\in \mathcal{H}_g| h \text{ factors through } N_{\mathbb{R}}\}\]
is non-empty. If $Y^{+}$ is a connected component of $Y_N$ and
$\gamma K_n\in L(\mathbb{A}_f)/K_n$, the image of $Y^{+}\times
\{\gamma K_n\}$ in $A_g$ is an algebraic subvariety. We define a
\emph{special (Shimura) subvariety} as an algebraic subvariety $S$ of $A_g$
which arises in this way, i.e., there exists a connected component
$Y^+\subset Y_N$ and an element $\gamma K_n\in
L(\mathbb{A}_f)/K_n$
such that $S$ is the image of $Y^{+}\times \{\gamma K_n\}$ in $A_g$.\\

By sending a point $t\in T$ to the class of the pair
$(C_t,\eta_t)$ one gets a map $T\to R_g$ and one can show that the
dimension of the image $T(\tilde{G},\tilde{\theta},\sigma)$ of
$T$ (which we also denote by $T$) in $R_g$ is equal to $s-3$, see \cite{CFGP}, p. 7. Therefore the
above family gives rise to a subvariety of $R_g$ of dimension
$s-3$. In this paper, we are interested in determining whether the
subvariety $Z=\overline{\mathscr{P}(T)}\subset A_{g-1}$ is a special or Shimura subvariety. The analogous statements hold also 
for ramified Prym maps and analogously we are interested in the image $Z=\overline{\mathscr{P}(T)}\subset A_{g}$ for $T\hookrightarrow R_{g,2}$ again of dimension $s-3$.\\

\begin{remark} \label{smallestshimura}
The above constructions, make it clear that there is a
$\mathbb{Q}$-variation of Hodge structures over $T$ with fibers
given by $H^1(\tilde{C_t},\mathbb{Q})_-$. We choose a
Hodge-generic point $t_0\in T(\mathbb{C})$ and let $M\subset
GL(H^1(\tilde{C_{t_0}},\mathbb{Q})_-)$ be the generic Mumford-Tate
group of the family. Let $S_{f}$ be the natural Shimura variety
associated to the reductive group $M$. So in general, this subvariety is different from the one with the same notation
in \cite{MZ}, Remark 2.7. The special subvariety $S_{f}$ is in fact the
\emph{smallest} special subvariety that contains $Z$ and its
dimension depends on the real adjoint group $M^{ad}_{\mathbb{R}}$.
Indeed, if $M^{ad}_{\mathbb{R}}= Q_{1}\times...\times Q_{r}$ is
the decomposition of $M^{ad}_{\mathbb{R}}$ to $\mathbb{R}$-simple
groups then $\dim S_{f}= \sum \delta(Q_{i})$. If
$Q_{i}(\mathbb{R})$ is not compact then $\delta(Q_{i})$ is the
dimension of the corresponding symmetric space associated to the
real group $Q_{i}$ which can be read from Table V in \cite{H}. If
$Q_{i}(\mathbb{R})$ is compact (in this case $Q_{i}$ is called anisotropic) we
set $\delta(Q_{i})=0$. We remark that for $Q=PSU(p,q)$,
$\delta(Q)=pq$ and for $Q=Psp_{2p}$, $\delta(Q)=\frac
{p(p+1)}{2}$. Our computations below show that in fact such
factors do occur in the decomposition of $M$, see Lemma~\ref{dimeigspace}. Note that $Z$ is a
Shimura subvariety if and only if $\dim Z=\dim S_{f}$, i.e.,
if and only if $Z=S_{f}$.
\end{remark}

These observations lead to the following key lemma.

\begin{lemma} \label{dimS_f}
If $\dim S_f>s-3$, then the family does not give rise to a special
subvariety of the Prym locus.
\end{lemma}

\begin{proof}
By the constructions and explanations in previous paragraphs, we
have a map $\mathscr{P}:R_g\to A_{g-1}$ (resp. $R_{g,2}\to A_g$)
and $Z=\overline{\mathscr{P}(T)}\subset A_{g-1}$ (resp. $A_{g}$).
Now unlike the Torelli map, the Prym map is not injective,
however, it holds that $\dim Z\leq s-3$. Hence if $\dim S_f>s-3$,
one concludes that $Z\neq S_f$ and therefore $Z$ is not a special
subvariety by the above.
\end{proof}

In the light of the above lemma, our strategy is to show that for families with large $s$, the subvariety $S_f$ constructed above
is of dimension strictly greater than $s-3$, hence the subvariety $Z$ is not special by the above lemma. 

\subsection{Calculations related to abelian covers}

\begin{remark} \label{abeliangroupcharacter}
Let $G$ be a finite abelian group,
then the character group of $G$, $\mu_G= Hom(G,\mathbb{C}^{*})$ is isomorphic to $G$. To see
this, first assume that $G=\mathbb{Z}/N$ is a cyclic group. Fix
an isomorphism between $\mathbb{Z}/N$ and the group of $N$-th
roots of unity in $\mathbb{C}^{*}$ via $1\mapsto exp(2\pi i/N)$.
Now the group $\mu_G$ is isomorphic to this latter group via $\chi
\mapsto \chi(1)$. In the general case, $G$ is a product of finite cyclic groups, so this isomorphism extends to an
isomorphism $\varphi_G: G \xrightarrow{\sim} \mu_G$. In the sequel, we use this isomorphism frequently to identify
elements of $G$ with its characters.
\end{remark}

Let $n\in
G\subseteq (\mathbb{Z}/N)^{m}$ or equivalently the corresponding character by the above remark. 
We regard $n$ as an $1\times m$ matrix. Then we can form the matrix product $n\ldotp A$ and
$n\ldotp A=(\alpha_{1},...,\alpha_{s})$. Here, as usual, the operations are
taking place in $\mathbb{Z}/N$ but the $\alpha_{j}$ are regarded as
integers in $[0,N)$.\\

Let $n\in G$ be the element $(n_1,\cdots, n_m)\in
G\subset (\mathbb{Z}/N\mathbb{Z})^{m}$. A basis for the
$\mathbb{C}$-vector space $H^0(C,K_C)$ is given by the forms
$\omega_{n,\nu}=z^{\nu} w_{1}^{n_1}\cdots w_{m}^{n_m}\displaystyle
\prod_{j=1}^{s} (z-z_j)^{\lfloor
-\frac{\tilde{\alpha_j}}{N}\rfloor}dz$. Here $n\in G, \tilde{\alpha_j}$ is the lift of $\alpha_j$ to $[0,N)$ and $0\leq
\nu\leq d_{n}=-1+\displaystyle \sum_{j=1}^{s}\langle-
\frac{\alpha_{j}}{N}\rangle$. The fact that the above elements
constitute a basis can be seen in \cite{MZ}, proof of Lemma 5.1, where
the dual version for $H^1(C,\mathcal{O}_C)$ is proved. Note that in \cite{MZ}, Proposition 2.8, the 
formula for $d_n$ has been proven using a different method.\\

The following lemma is key to our later analysis and shows that we can use the dimension $d_n$ of the 
eigenspace $H^0(C,K_C)_n$ (computed in \cite{MZ}, Proposition 2.8) in our computations with $H^0(C,K_C)_{-,n}$. 
We remark that if $n=(n_1,\cdots, n_m)\in G\subset (\mathbb{Z}/N\mathbb{Z})^{m}$, we consider the $n_i\in [0,N)$ and their sum as
integers.

\begin{lemma} \label{dimeigspace}
The group $G$ acts on the spaces $H^0(C,K_C)_{-}$ and for $n\in G$, it holds that
$H^0(C,K_C)_{-,n}=H^0(C,K_C)_{n}$ if and only if $n_1+\cdots+n_m$ is odd and 
$H^0(C,K_C)_{-,n}=0$ otherwise. Similar statements hold for $H^1(C,\C)_{-,n}$.

\end{lemma}

\begin{proof}
By the construction of an abelian cover of $\mathbb{P}^1$ in the beginning of this section, the action of $\sigma$ is given by $w_i\mapsto -w_i$ for some subset of $\{1,\cdots, m\}$
(and naturally $w_j\mapsto w_j$ for $j$ in the complement of this subset).  Using the basis of the 
space $H^0(C,K_C)$ given above, the $\sigma$-eigenspace $H^0(C,K_C)_+$ is the set of all
$\omega_{n,\nu}$ with $n_1+\cdots+n_m$ even and $H^0(C,K_C)_-$ is the set of all $\omega_{n,\nu}$ with $n_1+\cdots+n_m$ odd. If $n\in G$
(or the corresponding character in $\mu_G$, see Remark~\ref{abeliangroupcharacter}) the eigenspace $H^0(C,K_C)_{-,n}$ is then given by:

\[H^0(C,K_C)_{-,n}=\begin{cases}
        0, & \text{if } 2|n_1+\cdots+n_m,\\
        H^0(C,K_C)_{n} & \text{otherwise,} 
        \end{cases} \]

and in general, it holds that:
\[H^1(C,\C)_{-,n}=\begin{cases}
        0, & \text{if } 2|n_1+\cdots+n_m,\\
        H^1(C,\C)_{n} & \text{otherwise.} 
        \end{cases} \]        
\end{proof}
We show that if $s$ is large enough, then  there
exist $\mathbb{R}$-simple factors $Q_i$, such that
$\sum\delta(Q_{i})>s-3$ and hence the family is not Shimura.

\begin{remark} \label{eigenspacetype}
 Let $f:Y\rightarrow T$ be a family of
 abelian Galois covers of $\mathbb{P}^{1}$ as constructed in
 section $1$. Then the local system
 $\mathcal{L}=R^{1}f_{*}\mathbb{C}_-$ gives rise to a polarized
 variation of Hodge structures (PVHS) of weight $1$ whose fibers are the HS discussed above. Consider the
 associated monodromy representation $\pi_{1}(T,x)\to GL(V)$, where
 $V$ is the fiber of $\mathcal{L}$ at $x$. The Zariski closure of the
 image of this morphism is called the \emph{monodromy group} of $\mathcal{L}$.
 We denote the identity component of this group by $Mon^{0}(\mathcal{L})$. The
 PVHS decomposes according to the action of the abelian Galois
 group $G$ and the eigenspaces $\mathcal{L}_{i}$ (or
 $\mathcal{L}_{\chi}$ where $i\in G$ corresponds to character $\chi
 \in \mu_{G}$ by Remark~\ref{abeliangroupcharacter}) are again variations of Hodge
 structures and we are mainly interested in these. Take a $t\in T$
 and assume that $h^{1,0}((\mathcal{L}_{i})_{t})=a$ and
 $h^{0,1}((\mathcal{L}_{i})_{t})=b$.  The above computations show how to calculate $h^{1,0}((\mathcal{L}_{i})_{t})$
 (resp. $h^{0,1}((\mathcal{L}_{i})_{t})$). Since monodromy group respects
 the polarization of the Hodge structures (\cite{R}, 3.2.6),
 $(\mathcal{L}_{i})_{t}$ is equipped with a Hermitian form of
 signature $(a,b)$ (see \cite{DM}, 2.21 and 2.23). This implies that
 $Mon^{0}(\mathcal{L}_{i}) \subseteq U(a,b)$. In this case, we say
 that $\mathcal{L}_{i}$ is \emph{of type} $(a,b)$. Lemma~\ref{dimeigspace} (together with \cite{MZ}, Proposition 2.8) computes type of any eigenspace.
 Two eigenspaces $\mathcal{L}_{i}$ and $\mathcal{L}_{j}$ of types $(a,b)$ and $(a^{\prime},b^{\prime})$ are said to be \emph{of
 distinct types} if $\{a,b\}\neq \{a^{\prime},b^{\prime}\}$.  We call an eigenspace $\mathcal{L}_{i}$ \emph{trivial} if it is of type $(a,0)$ or $(0,b)$.
 \end{remark}
 
 \begin{remark} \label{monodromydecomp}
 Let $\mathcal{V}$ be a variation of Hodge structures over a non-singular connected complex algebraic variety. 
 If there is a point $s$, such that the Mumford-Tate group $MT_s$ is 
 abelian, then the connected monodromy group is a normal subgroup of the generic Mumford-Tate group $M$. In fact in this case $Mon^0=M^{der}$, see \cite{Andr}. In particular, 
If $Z\subset A_g$ is special, then $Mon^0=M^{der}$. Consequently, if the family $f\colon Y\to T$ gives rise to a Shimura subvariety
 and $M^{ad}_{\mathbb{R}}=\prod_{1}^{l} Q_{i}$ as a product of simple Lie groups then $Mon^{0,ad}_{\mathbb{R}}=\prod_{i\in K} Q_{i}$ for some $K\subset \{1,\cdots, l\}$.
\end{remark}
 
\section{Shimura families in the Prym locus}

In this section, we show that for large $s$, families of Pryms of abelian covers of the line do not
give rise to special subvarieties in $A_g$. We prove several results in this direction, which show that
the special families of abelian covers are very limited.

\subsection{Strategy of the proof of main theorems} \label{strategy}
The following argument will be used in all of the theorems below: 
In order to prove that the variety $Z$ is not special, or in other words that $\dim Z< \dim S_f$, we assume on the contrary that $Z$ is special. So $Mon^0$ will be a normal subgroup of the generic MT-group $M$
by Remark~\ref{monodromydecomp}. In particular by this assumption $Mon^0(\mathcal{L}_{i})$ also appear in $M^{ad}$. According to Remark~\ref{monodromydecomp} above, if $M^{ad}_{\mathbb{R}}=\prod_{1}^{l} Q_{i}$ is the decomposition
 into $\mathbb{R}$-simple factors, then $Mon^{0,ad}_{\mathbb{R}}=\prod_{i\in K} Q_{i}$ for some subset $K$ as well. We need to find eigenspaces
 $\mathcal{L}_{j_i}$ of distinct types $\{a_i,b_i\}$
 with $a_i$ and $b_i$ large enough. More precisely, if we can find eigenspaces $\mathcal{L}_{j_i}$ as above, then these eigenspaces, being of different types $\{a_i,b_i\}$, 
 give rise to non-isomorphic $Q_i=Mon^{0,ad}_{\mathbb{R}}(\mathcal{L}_{j_i})=PSU(a_i,b_i)$ for
 $i\in K$ in the above decomposition of $Mon^{0,ad}_{\mathbb{R}}$, see Remarks ~\ref{smallestshimura} and ~\ref{eigenspacetype}, and if for
 $\delta(\mathcal{L}_{j_i})=\delta(Mon^{0,ad}_{\mathbb{R}}(\mathcal{L}_{j_i}))$, we have that 
 $\sum \delta(\mathcal{L}_{j_i})>s-3$, then
 $\dim S_f\geq \sum \delta(\mathcal{L}_{j_i})>s-3$. This is what being \emph{large enough} means in the above, note that $\delta(\mathcal{L}_{j_i})$ 
 depends in our examples only on $a_i$ and $b_i$; see Remark~\ref{smallestshimura}.

Now we are ready to state our main results. In \cite{CFGP}, it is shown that for families of unramified Prym varieties $\tilde{G}$ is not cyclic. The following
result shows that in general, the cyclic families which can be special are of a very restrictive nature.

\begin{proposition}
Let $(\tilde{G},\tilde{\theta},\sigma)$ be a Prym datum of a family of abelian covers with $\tilde{G}$ cyclic.
Then either all of the eigenspaces $\mathcal{L}_a$ except one are trivial or all of them are of the same type 
(necessarily of type $(1,s-3)$) or the family is not special.
\end{proposition}

\begin{proof}
Let $\tilde{G}=\mathbb{Z}/N$. Using Lemma~\ref{dimeigspace} and \cite{M}, Lemma 7.2, there exists an $n\in (\mathbb{Z}/N)^*$ with 
$\{d_n,d_{-n}\}\neq \{0,s-2\}$. If all other eigenspaces are trivial the claim is proved. Otherwise, assume that there exists another eigenspace corresponding to $n^{\prime}$
such that $\mathcal{L}_n$ and $\mathcal{L}_{n^{\prime}}$ are of distinct types and $\mathcal{L}_{n^{\prime}}$ is not trivial. Since $d_n+d_{-n}=s-2$, it follows
that $d_nd_{-n}\geq s-3$ and by assumption, $d_{n^{\prime}}d_{-n^{\prime}}>0$. Then the argument in ~\ref{strategy} shows that $\dim S_f>s-3$ and hence the family is not special.

\end{proof}

\begin{proposition} \label{cyclicprym}
Let $(\tilde{G},\tilde{\theta},\sigma)$ be a Prym datum of a family of abelian covers with $\tilde{G}$ cyclic
corresponding to $(a_1,\cdots, a_s)$ with $\sum a_i>2N, \sum [-a_i]>2N$. If $s>5$, Then the family of Pryms is not special.
\end{proposition}

\begin{proof}
Assume that $Z$ is special. We are going to wield the observation in \S ~\ref{strategy} as in the last proposition. Suppose that the single row is of type
$(l,s-l-2)$ so $\delta(\mathcal{L}_1)=(l,s-l-2)$ by the notation of Remark~\ref{smallestshimura}. 
Using Lemma~\ref{dimeigspace} and Remark~\ref{smallestshimura}, and the assumption implies that 
\[\dim S_f\geq l(s-l-2)\geq 2(s-4)>s-3,\]
which is a contradition. Hence the family is not special.
\end{proof}

The following theorem generalizes the above result to the case of arbitrary abelian covers.

\begin{theorem} \label{abeigprym}
Let $(\tilde{G},\tilde{\theta},\sigma)$ be a Prym datum of a family of abelian covers.
If there exist eigenspaces of distinct types $(a,b)$ and $(c,d)$ with $a,b,c,d\geq 2$ and the corresponding row vectors contain less than $s$ zero entries,  
then for $s>13$, the family of Pryms is not special. 
\end{theorem}

\begin{proof}
Proceeding as in \S~\ref{strategy}, our strategy is to assume that $Z$ is special and show that this leads to a contradiction.
We first remark that if the corresponding row vectors contain respectively $l$ and $k$ zeros, then $0\leq l+k\leq 2s-2$.
Now suppose that the eigenspaces are of types $(l_1,s-l-l_1-2)$ and $(k_1,s-k-k_1-2)$. Note that $l>0$ (resp. $k>0$) if and only if in the row vector
corresponding to the eigenspaces there are some zero entires. By the same arguments as in Proposition~\ref{cyclicprym}, the assumptions of the theorem then imply that
\[\dim S_f\geq l_1(s-l-l_1-2)+k_1(s-l-k_1-2)\geq 2(s-l-4)+2(s-k-4)=4s-2(l+k)-16.\]
However, since $l+k<s$, it follows that for $s>13$, $\dim S_f>s-3$, which contradicts our assumption. Hence the claim is proved.
\end{proof}

The above theorem is particularly interesting if our abelian covers have 2 rows, i.e. $m=2$ or the abelian group has the form $\tilde{G}=\mathbb{Z}/N\times \mathbb{Z}/M$.

\begin{corollary}
Let $(\tilde{G},\tilde{\theta},\sigma)$ be a Prym datum of a family of abelian covers.
If there exist eigenspaces of distinct types $(a,b)$ and $(c,d)$ with $a,b,c,d\geq 2$. Then for $s>13$, the family of Pryms is not special. 
\end{corollary}

\begin{proof} 
If the family has only two rows, then obviuosly, $l+k<s$ and we can apply the above Theorem~\ref{abeigprym}.

\end{proof}

\end{document}